\documentclass[11pt,reqno,a4paper]{amsart}
\usepackage{graphicx}
\usepackage{subcaption}
\usepackage{enumerate} \usepackage[colorlinks=true, pdfstartview=FitV,
linkcolor=blue, citecolor=red, urlcolor=blue]{hyperref}
\usepackage{amsmath,amsfonts,latexsym,amssymb}
\usepackage[utf8]{inputenc} 
\usepackage[T1]{fontenc}
\usepackage{braket}
\usepackage{ae,aecompl}
\usepackage{dsfont} 
\usepackage{pxfonts}
\usepackage{microtype}
\usepackage{ae,aecompl}
\usepackage{marginnote}
\usepackage{accents}
\newtheorem{theorem}{Theorem}[subsection]
\newtheorem{lemma}[theorem]{Lemma}
\newtheorem{proposition}[theorem]{Proposition}
\newtheorem{corollary}[theorem]{Corollary}

\newtheorem{definition}[theorem]{Definition}

\newcommand{\T}{\ms T}
\renewcommand{\leq}{\leqslant}
\renewcommand{\geq}{\geqslant}

\newcommand{\rep}{\operatorname{Rep}}

\renewcommand{\epsilon}{\varepsilon}
\newcommand{\sbt}{\,\begin{picture}(-1,1)(-1,-1)\circle*{2}\end{picture}\ }
\renewcommand{\dot}[1]{\overset{\sbt}{#1}}

\renewcommand{\hom}{\operatorname{Hom}}

\makeatletter
\newcommand{\oset}[3][0ex]{%
  \mathrel{\mathop{#3}\limits^{
    \vbox to#1{\kern-2\ex@
    \hbox{$\scriptstyle#2$}\vss}}}}
\makeatother

\newcommand{\sld}{\mathsf{SL}_2(\mathbb R)}

\newcommand{\ms}{\mathsf}

\newcommand{\cE}{{\mathcal E}}
\newcommand{\defeq}{\coloneqq}

\newcommand{\sop}{{\mathsf{SO}(p,p)}}

\newcommand{\diffo}{\left.\frac{\rm d}{\rm d t}\right\vert_{t=0}}

\newcommand{\ext}{\mathsf{\Lambda}}
\newcommand{\cV}{\mathcal V}

\newcommand{\flo}[1]{\left(#1_t\right)_{t\in\mathbb R}}

  \newcommand{\grf}{\pi_1(\Sigma)}
  \newcommand{\bgrf}{\partial_\infty\grf}
  \newcommand{\bL}{\mathbf L}
  \newcommand{\bF}{\mathbf F}
  \title{Entropy and affine actions for surface groups}
  \author{François Labourie}
\begin{document}
  \maketitle
    \begin{abstract}
    	We give a short and independent proof of a theorem of Danciger of Zhang: surface groups with Hitchin linear part cannot act properly on the affine space. 
    	The proof is fundamentally different and relies on ergodic methods.
    \end{abstract}   
\subsection{Introduction} Our goal is to give an independent proof, based on thermodynamical ideas, of a recent theorem by Danciger and Zhang \cite{Danciger:2018tl}.
  \begin{theorem}\label{theo:aff}
	Assume that a surface group acts on the affine space so that its linear part is a Hitchin representation. Then its action on the affine space is not proper.
\end{theorem}

The proof given in this article, which also gives results  of independent  interest (Theorems \ref{theo:smooth} and \ref{theo:entrop}) , uses ergodic theory and hyperbolic dynamics: entropy, Sinai--Ruelle--Bowen measures. We hope that the use of this type of methods open a novel approach on the study of proper affine actions, expanding previous work of Goldman Margulis and the author.

Being very optimistic, as an approach to Auslander conjecture,  one could hope that, in the spirit of Kahn--Markovic \cite{Kahn:2009wh} and Kahn--Labourie--Moze \cite{Kahn:2018wx}, the presence of free groups could help in building surfaces groups close to be Fuchsian inside groups acting cocompactly on the affine space.

A {\em surface group} is the fundamental group of a closed connected oriented surface of genus greater than 2.
A {\em Hitchin representation} \cite{Hitchin:1992es} is a representation that can be deformed into a {\em Fuchsian representation}, that is a discrete representation with values in an irreducible $\mathsf{SL}(2,\mathbb R)$. 

A conjecture, attributed to Auslander \cite{Auslander:1964fs}, states that if a group  $\Gamma$ acts properly and cocompactly on the affine space then it does not contain a free group. This conjecture has been proved up to dimension 7 by Abels, Margulis and So\"{\i}fer in \cite{Abels:2012wd}. On the other hand, Margulis in \cite{Margulis:1983a} has exhibited free groups acting properly on the affine space.  A work  of  Goldman, Margulis and the author \cite{Goldman:2008}, further extended by Ghosh and Treib \cite{Ghosh:yPvRLZo8}, have shown how to characterise proper actions of a hyperbolic group using the {\em Labourie--Margulis diffusion}, which is an extension to measures -- introduced  in \cite{Labourie:2001} --  of the {\em Margulis invariant} introduced by Margulis in \cite{Margulis:1984}. As for surface groups, there were shown by Mess \cite{Mess:2007kx}  to admit  no proper affine actions on the affine 3-space. An alternate proof was given by Goldman and Margulis \cite{Goldman:2000a} and then by the author  \cite{Labourie:2001} with the extension to groups whose linear part is Fuchsian. On the other hand, Danciger, Guéritaud  and Kassel \cite{Danciger:2018va} exhibited examples of proper affine actions of surface groups, or more generally some Coxeter groups,  in higher dimensions. For a survey and similar considerations  see \cite{Smilga} and \cite{survey}

I thank Sourav Ghosh,  Fanny Kassel, Andres Sambarino, Ilia Smilga and Tengren Zhang for valuable discussions and input. 

\subsubsection{A sketch of the proof}
As an initial observation, we observe  that the problem reduces to the case of representations whose linear part is in $\mathsf{SO}(p,p-1)$.  Indeed,  according to Guichard, the Zariski closure $\mathsf G$ of a Hitchin  representation always contains the irreducible $\sld$, and, if non Zariski dense, contained in either $\mathsf{Sp}(2p)$ or $\mathsf{SO}(p,p-1)$. Recall finally that if an element of the affine group acts properly on the affine space, then 1 is an eigenvalue of its linear part. Thus 1 is in the spectrum of any element in the Zariski closure of its linear part. It follows that the representation is in odd dimensions and non Zariski dense in $\mathsf{SL}(2p-1)$, thus contained in $\mathsf{SO}(p,p-1)$. 

After this initial observation, the proof follows the thermodynamic theme introduced in \cite{Labourie:2001}.  A sketch is as follows

From now on, let $\Gamma$ be a surface group whose linear part is a Hitchin representation in $\mathsf{SO}(p,p-1)$.  The {\em Labourie-Margulis diffusion $M$}  is a continuous function on the space of  measures  invariant  by the geodesic flow of the surface, associated to  the representation on the affine space \cite{Labourie:2001}. According to a generalisation of \cite{Labourie:2001,Goldman:2008} due to Ghosh and Treib \cite[Theorem 7.1 and Definition 4.4]{Ghosh:yPvRLZo8}, if there exists a measure   $\mu$ so that $M(\mu)=0$, then the action on the affine space is not proper.

As a first step in the proof,  we embed  the Lie algebra  of ${\mathbb R}^{p,p-1}\rtimes \mathsf{SO}(p,p-1)$ as a subalgebra of $\sop$. Thus an affine representation is seen as a deformation of the linear part of the representation in $\sop$. As in \cite{Goldman:2000a,Ghosh:2018wx}, we now interpret in lemma \ref{lem:marg} the Margulis invariant as a variation of the $p^{\tiny th}$ eigenvalue (or the $(p+1)^{\tiny th}$), while the other eigenvalues remain constant. 

As a consequence of the Abramov lemma and the definition of equilibrium states as done in \cite{Sambarino:2014jv}, we can now interpret, in lemma \ref{lem:abra}, the Margulis invariant as the variation of the topological entropy of the {\em last root flow},  a flow for which the length of the closed orbit associated to $\gamma$ is the logarithm of the product of the $(p-1)^{\tiny th}$ and $p^{\tiny th}$ eigenvalue of $\rho(\gamma)$.

A recent series of results by Pozzetti, Sambarino and Weinhard \cite{Pozzetti:2019uc} implies among other things that this entropy is constantly equal to 1. We prove this result independently in Theorem \ref{theo:entrop} by proving that the {\em isotropic limit curve} is smooth and use an idea due to Potrie--Sambarino \cite{Potrie:2014uta} to obtain the same result. This is a parallel  to \cite[Theorem 9.9]{Pozzetti:2019uc}.

This smoothness, obtained in Theorem \ref{theo:smooth} now follows from a general lemma about proximal bundles -- lemma \ref{lem:prox} -- and a transversality property -- proposition \ref{pro:trans} -- that we prove for Fuchsian representations in $\mathsf{SO}(p,p)$. This transversality property is a consequence of Lusztig positivity \cite{Lusztig:1994} as used in \cite{Fock:2006a} and  we wonder whether this property could characterise Hitchin representations in $\mathsf{SO}(p,p)$ within Anosov representations.

Combining these simple ideas, on obtains that the Margulis invariant for the Bowen--Margulis measure of the last root flow is zero and thus concludes the proof of  the Theorem by Danciger and Zhang.

\subsection{Isotropic flags and the geometry of $\sop$}
Let $E$ be a vector space equipped with a metric $Q$ of signature $(p,p)$, let $\sop$ be its isometry group. For every vector space $V$ in $E$ we denote by $V^o$ its orthogonal with respect to the quadratic form. An {\em isotropic space} is a vector space on which the restriction of $Q$ vanishes, a {\em maximal isotropic plane} is an isotropic plane of dimension $p$. We denote by $\bL$ the space of maximal isotropic planes.

Recall that the action of $\sop$ on $\bL$ has two orbits, which are both connected components of $\bL$. To distinguish them, let us fix a spacelike $p$-plane $F$, and an orientation on $F$ and $F^o$.
Any $p$-isotropic plane $P$ is then the graph of a linear  isomorphism $A$ from $F$ to $F^o$. We say $P$ is {\em positive} when $A$ preserves the  orientation and {\em negative} otherwise. We denote by $\bL^+$ the space of positive $p$-isotropic planes and $\bL^-$ the space of negative $p$-isotropic planes. Any $(p-1)$-isotropic plane is contained in exactly one positive isotropic $p$-plane and one negative isotropic $p$-plane.

 An {\em isotropic flag} is a collection of isotropic planes $L=(L_i)_{1<i\leq p}$ so that $L_i\subset L_{i+1}$, $\dim(L_i)=i$. An isotropic flag $L$ can be positive or negative depending on $L_p$.  We denote by $\bF$  the space of positive isotropic flags. The group $\sop$ acts transitively on $\bF$ and the stabiliser of a point is the minimal parabolic subgroup of $\sop$.  Observe also that $L_p$ is determined by $L_{p-1}$.  Two isotropic flags $L$ and $M$ are {\em transverse} is for all $i$, we have $M_i\oplus L_i^{\circ}=L_i\oplus M_i^{\circ}=E$.
 
 A $p$-tuples of lines $E=(E_i)_{i=1,\ldots,p}$ is {\em isotropic} if $E_1+\ldots + E_p$ is maximal isotropic. The isotropic flag $F(E)$ {\em associated to $E$} is $
 F(E)=(L_1,\ldots, L_p)$ where $L_i=E_1+\ldots+ E_i$.
 Two $p$-tuples of lines $E=(E_i)_{i=1,\ldots,p}$ and  $\overline E=(\overline E_i)_{i=1,\ldots,p}$ are {\em $Q$-paired}  if they are both isotropic and  $Q$ restricted to $E_i \oplus \overline E_j$ is  zero for $i\not=j$ and non degenerate otherwise.
 We then have

\begin{proposition}{\sc [Transverse flags]} \label{pro:transv} The maps that sends  $(E,\overline E)$ to $(F(E),F(\overline E)$  is an $\sop$-equivariant bijection from the space of   $Q$-paired $p$-tuples of lines to the set of transverse flags.
\end{proposition}

Let us conclude with a description of the tangent space to $\bL$: 
Let $E_0$, $E_1$  be two transverse isotropic planes. Let $F$ be a $p$-plane transverse to  to $E_1$, so that  $F$ is the graph of $f\in \hom(E_0,E_1)$. Let $\omega_F$ be the 2-form on $E_0$ given by $\omega_F(u,v)=Q(u,f(v))$.

\begin{proposition}{\sc [Identification]}\label{pro:tgiso}
	Let $\theta_0$ and $\theta_1$ be two transverse isotropic planes. The map $F\to \omega_F$ is a diffeomorphism between the space of isotropic planes transverse to $\theta_1$ and $\ext^2(\theta_0^*)$. 
	In particular, $\T_{\theta_0}\bL=\hom(\theta_0,\theta_1)$ identifies with $\ext^2(\theta_0^*)$. 
\end{proposition}
\subsection{Anosov representations for ${\sop}$ and $\mathsf{SO}(p,p-1)$}
Let $\Sigma$ be a closed hyperbolic surface, $X$ its unitary tangent bundle and $\flo{\varphi}$ its geodesic flow. We also write  $\Gamma\defeq \pi_1(\Sigma)$.

Let $\rho$ be a representation of $\Gamma$ in ${\sop}$  that we see acting on a vector space $E$ equipped with a quadratic form $\braket{\mid}$ of signature $(p,p)$. We denote by $\cE_\rho$ the associated flat bundle on $X$ and $E_x$, the fibre of $\cE_\rho$ at a point $x$ in $X$.

Observe that $\flo{\varphi}$ lifts to a flow $\flo{\Phi}$ acting on $\cE_\rho$ by vector bundle automorphisms which are parallel along the geodesic flow
\begin{definition}{\sc[Anosov representations for ${\sop}$]}\label{def:ano}
We say $\rho$ is {\em  Borel Anosov for $\mathsf{SO(}p,p)$}, if the bundle $\cE_\rho$ splits into $2p$ continuous  line  bundles $\cE_{i}$, $\overline \cE_i$ with $1\leq i\leq p$,  with the following properties
\begin{enumerate}
	\item The lines bundles $\cE_i$ and $\overline \cE_i$ are invariant under $\flo{\Phi}$ and $Q$-paired,
	\item The flow $\flo{\Phi}$ contracts the bundles 
\begin{eqnarray}
	\cE_i^*\otimes \overline{\cE_p}\otimes \cE_i &  &\hbox{ when $i<p$, }\\
	\cE_i^*\otimes \cE_j & & \hbox{ when $j<i$,} \ .
\end{eqnarray}

\end{enumerate}
\end{definition} We recall that a flow $\flo{\Phi}$ {\em contracts  on a bundle $E$} over a compact manifold if  there exists a continuous metric  and  positive constants $a$ and $b$, so that for all positive $t$, 
$\Vert \Phi_{t}u\Vert\leq ae^{-bt}\Vert u\Vert$.
To be contracting on a compact manifold is independent on the parametrisation of the flow or the choice of the metric.

Let $C_\gamma$ be a closed orbit of the flow on $X$ of length $\ell_\Gamma$ associated to an element $\gamma$ in $\Gamma$. Then $\rho(\gamma)$ is conjugated to the endomorphism $\Phi_{\ell_\gamma}$ of $E_x$, and in particular $(\cE_i)_x$ and $(\overline \cE_i)_x$  are eigenlines of  $\Phi_{\ell_\gamma}$. We denote by $\lambda_i\rho(\gamma)$  and $\overline\lambda_i\rho(\gamma)$  the  corresponding eigenvalues, which are also eigenvalues of $\rho(\gamma)$. The Anosov condition gives $\lambda_i\cdot\overline\lambda_i=1$ and the following ordering of eigenvalues
$$
\lambda_1>\lambda_2>\ldots >\sup\{\lambda_p,\lambda_p^{-1}\}\geq \inf\{\lambda_p,\lambda_p^{-1}\}>\ldots >\overline\lambda_1\ .
$$
\subsubsection{Limit curves} Let $\rho$ be an Anosov representation for $\sop$. 
We may lift the bundle $E_\rho$ to a trivial bundle over the unitary tangent bundle $Y$ of the hyperbolic bundle. The line bundles $\cE_i$ and $\cE_i$ also lifts and since they are parallel under $\flo{\Phi}$. Let then consider the maps
\begin{eqnarray*}
	E_i:(x,y)\mapsto E_i(x,y)\defeq(\cE_i)_z\ , \ \
	 \overline E_i:(x,y)\mapsto \overline E_i(x,y)\defeq({\overline\cE}_i)_z\ ,
\end{eqnarray*}
where $z$ in a point in the geodesic defined by the pair of distinct points $(x,y)$ in the boundary at infinity $\partial_\infty \mathbf H^2$ of the hyperbolic plane $\mathbf H^2$.
\begin{proposition}{\sc [Limit curve]}\label{pro:limi} We have 
$
E_q(x,y)\defeq \overline E_{q}(y,x)
$.
Moreover the isotropic flag $\xi(x,y)$ given by $(E_1(x,y), E_2(x,y),\ldots E_p(x,y))$ only depends on $x$. 
\end{proposition}
The map $\xi:x\mapsto\xi(x)\defeq\xi(x,y)$ is {\em the limit curve} of the Anosov representation.
\begin{proof} By density, it is enough to check  the first  identity for $(x,y)=(\gamma^+,\gamma^-)$ where $\gamma^+$ and $\gamma^-$ are respectively the attractive and repulsive points of an element  $\gamma$ of $\Gamma$. The result follows by the identification of $E_i$ with eigenlines of $\rho(\gamma)$. Similarly, for the second identity we know that $\xi(\gamma^-,\gamma^+)$ is an attractive point of $\rho(\gamma)$. It follows that $\xi(\gamma^-,y)=\rho(\gamma)^n\xi(\gamma^-,\gamma^{n}y)$. Since $\gamma^n(y)\rightarrow_{n\to\infty}=\gamma^+$. It follows that if $y\not=\gamma^-$
 $$\xi(\gamma^-,y)=\lim_{n\to\infty}\rho(\gamma)^n\xi(\gamma^-,\gamma^{n}y)=\xi(\gamma^-,\gamma^+)\ .$$
 This concludes the proof
\end{proof}
Using  proposition \ref{pro:transv}, we can recover the maps $E_i$ using the limit curve $\xi$. Let us finally define the {\em isotropic limit curves} $\Theta$ and $\overline\Theta$ from $\bgrf$ to $\bL$ as  
\begin{eqnarray}
\ \ \Theta\defeq \bigoplus_{i=1}^{p}E_i\ &,&\ \ \overline\Theta\defeq \bigoplus_{i=1}^{p}\overline E_i\ .
\end{eqnarray}

\subsubsection{Hitchin representations in $\mathsf{SO}(p,p-1)$ }\label{sec:hitch} 
By \cite{Labourie:2006}, if $\rho$ is a Hitchin representation in $\mathsf{SL}_{2p-1}(\mathbb R)$, we have a decomposition of the associated bundle 
	$$\mathcal V_\rho=\mathcal V_1\oplus \ldots\oplus \mathcal V_{2p-1}\ , \\$$
	 such that the line bundles $\mathcal V_i$ are invariant by the flow and the flow contracts $\hom(\mathcal V_i,\mathcal V_j)$ for $i>j$. If furthermore the representation is with values in  $\mathsf{SO}(p,p-1)$, the flow preserves a quadratic form of signature $(p,p-1)$, $\mathcal V_p$ is a timelike trivial bundle equipped with a trivial action of the flow, while the other $\mathcal V_i$ are lightlike.
\begin{proposition}\label{anoano}
	Any Hitchin representation with values in $\mathsf{SO}(p,p-1)$ is Anosov for $\sop$.
\end{proposition}
\begin{proof}
	 Taking $\cE_\rho=\mathcal V_\rho\oplus \mathbf R$ -- where $\mathbf R$ is the trivial line bundle -- equipped with the product metric, we obtain the decomposition as wished by taking
for $i<p$, $\cE_i=\mathcal V_i$ and $\overline \cE_i=\mathcal V_{2p-i}$ and finally $\cE_p$ and $\overline \cE_p$ to be the lightlike lines in $\cV_p\oplus \mathbf R$. \end{proof}

 \subsubsection{The principal $\sld$-representations}
 In this section, we will give an explicit description of the map $E_i$ in the case of Fuchsian representations. Let \begin{eqnarray}
	A=\left(\begin{array}{rcl}1&z\\ 0&1\end{array}\right)\ , \ \ 
\Lambda=\left(\begin{array}{rcl} \lambda & 0\\ 0 &\lambda^{-1}\end{array}\right)\ . \label{def:ABL}
\end{eqnarray} 
Recall that the {\em $(2p-1)$-dimensional  irreducible representation} of $\sld$ preserves a quadratic form $\braket{\mid}$ of signature $(p,p-1)$. Moreover there exists a basis $\epsilon_1,\ldots,  \epsilon_{2p-1}$ so that, writing $\overline\epsilon_i\defeq\epsilon_{2p-i}$ and $\alpha_{k,m}\defeq\braket{A(\epsilon_k)\mid\overline{\epsilon_m}}$,  for  all $z\not=0$, 
\begin{eqnarray*}
	\braket{\epsilon_k\mid\overline\epsilon_m}=\delta_{k,m}&,&
\Lambda(\epsilon_m)=\lambda^{2p-2m}\epsilon_m\\
 \alpha_{k,m}\not=0\  \hbox{ if $m\geq k$}\ & ,& \alpha_{k,m}=0\ \hbox{ if $m<k$}\ .
 \end{eqnarray*} 
The {\em principal representation} of $\sld$  in $\sop$ is described as follows:  let $V$ be a vector space on which  $\sld$ acts irreducibly preserving a quadratic form of signature $(p,p-1)$; Let $(\epsilon_1,\ldots \epsilon_{2p-1})$ be the basis of  $V$ as above;  let $L$ be a  line generated by a vector $f$.  We  
 introduce now the base 
$(e_1,\ldots,e_{p},\overline e_1,\ldots, \overline e_{p})$ of $E\defeq V\oplus L$ where
\begin{eqnarray*}
\forall i<p, \ \ e_i=\epsilon_i\ ,\ \   \overline e_i\defeq\overline\epsilon_i\defeq\epsilon_{2p-i}& &
e_{p}=\epsilon_p- f\ ,\ \overline e_{p}=\epsilon_p+ f\ .
\end{eqnarray*}
Then $\sld$ preserves the quadratic form given in these coordinates by
$$
\braket{e_i\mid e_j}=\braket{\overline e_i\mid \overline e_j}=0\ , \ \ \braket{e_i\mid  \overline e_j}=\delta_{i,j}
$$
By convention, $(e_1,\ldots,e_p)$ generates a positive isotropic space.

\subsubsection{The Fuchsian representations in $\mathsf{SO}(p,p-1)$ and ${\sop}$}

Let  $\Sigma$ be equipped with a hyperbolic structure and $\bgrf$ is identified  with $\bf P^1(\mathbb R)$. 
 Let  $\rho$ be a {\em fuchsian representation} of $\Gamma$  in $\sop$ of the form $J\circ\nu$ where $\nu$ is a discrete representation of $\Gamma$ in $\sld$. Let for $i\leq p$,  the lines 
$E_i(x_0,y_0)$, respectively $\overline E_i(x_0,y_0)$ be generated by  $e_i$, respectively $\overline e_i$.

where $(x_0,y_0)=([1:0],[0:1])$ are elements of $\bgrf$. 
Then, since the stabiliser of $(x_0,y_0)$ is the group generated by $\Lambda$ and $\Lambda$ preserves $E_i(x_0,y_0)$ and $\overline E_i(x_0,y_0)$ , we define coherently
$$
E_i(Ax_0,Ay_0)\defeq A(E_i(x_0,y_0))\ ,\ \ \overline E_i(Ax_0,Ay_0)\defeq A(\overline E_i(x_0,y_0))\ .
$$
Then for all $x$ and $y$, 
$$
E_i(Ax,Ay)= A(E_i(x,y))\ ,\ ,\ \overline E_i(Ax,Ay)\defeq A(\overline E_i(x,y))\ .
$$
One now immediately checks the proposition
\begin{proposition}
	If $\Gamma$ is a Fuchsian group in $\mathsf{PSL}(2,\mathrm R)$, $J(\Gamma)$ is an Anosov representation for $\sop$, whose limit curve is $\xi(x)=F(E(x,y))$.
\end{proposition}
The following transversality property will play a crucial role in the sequel
\begin{proposition}{\sc [Transversality]}
	For all pairwise distinct of triple points $(x,y,z)$ in $\bgrf$ 	\begin{eqnarray}
\Theta(z) &\pitchfork& \left(E_{p}(x,y)\oplus (E^o_{p-1}(x,y)\cap  \overline\Theta(y))\right)  \label{eq:trans2}\ . 
 \end{eqnarray}
\end{proposition}
\begin{proof} 
 It is enough  to consider the case $x=[1:0]$, $y=[0:1]$ and  $z=[z:1]=A([0:1)]$ where $A$  is as in Equation \ref{def:ABL}. Let now   
\begin{eqnarray*}
	u&=& -b_pf + \sum_{m=1}^{p} b_m\epsilon_m\in \Theta(x)\ ,\\
\hbox{so that }\ \  A(u)&\in&  \Theta(z) \cap F(x,y)\ , \\ 
\hbox{ where } F(x,y)&\defeq&  (E_{p}(x,y)\oplus (E^o_{p-1}(x,y)\cap\overline\Theta(y))\ .
\end{eqnarray*}
 Recall now that  $F^o(x,y)$ is generated by
$\left\{\overline\epsilon_{1},\ldots,\overline\epsilon_{p-1},\overline \epsilon_{p+1},\right\}$. Thus  for $k\leq p+1$ and $k\not=p$, $
\braket{A(u)\mid \overline\epsilon_k}=0$, in other words
\begin{eqnarray*}
	 0=\sum_{m=1}^{k}\alpha_{m,k} b_{m}\ .
\end{eqnarray*}
The matrix corresponding to this system  is upper triangular with non zero coefficients, it follows that for all $1\leq m\leq p$, we have $b_k=0$. Thus $\Theta(z)\cap F(x,y)=\{0\}$. \end{proof}
\begin{corollary}\label{pro:trans} Let $\rho$ be a representation close to a Fuchsian representation. Then the transversality property \eqref{eq:trans2} holds 
	\end{corollary}
	\begin{proof} This follows from the continuity of limit curves as a dependence of the representation \cite{Guichard:2012eg,Bridgeman:2015ba} and the fact that $\Gamma$ acts cocompactly on the space of triple pairwise distinct points in $\bgrf$.
\end{proof}

\subsection{The isotropic limit curves and the Smoothness Theorem}
\begin{theorem}\label{theo:smooth}{\sc [Smoothness theorem]}
Let $\sop$ be  Anosov representation  satisfying the Transversality Property \eqref{eq:trans2},  then the image of the isotropic limit curve  $\Theta$ is a smooth curve $M$. Moreover, using the identification of proposition \ref{pro:tgiso}, 
$\T_{\Theta}M=\ext^2 (E^*_{p-1}\oplus E^*_{p})$. \end{theorem}

\subsubsection{Proof of the Smoothness Theorem \ref{theo:smooth}}
We will denote in general by $V_x$ the fibre at $x\in X$ of a vector bundle $V$ over a compact base $X$. Let $\flo{\varphi}$ be a flow on $X$ which lifts to a flow $\flo{\Phi}$ of bundle automorphisms on $V$. 

\begin{definition}{\sc [Proximal bundle]}
	We say the lift $\flo{\Phi}$ is {\em proximal} if there exists a continuous  $\flo{\Phi}$-invariant {\em proximal decomposition}  $V=Z\oplus W$  so that \begin{enumerate}
		\item the subbundle $Z$ has rank one,
		\item The flow contracts the subbundle $Z$ and 
	the bundle $Z^*\otimes W$,		
	\end{enumerate}
	\end{definition}
The following lemma is crucial in the smoothness part of the result.
\begin{lemma}\label{lem:prox}{\sc [Proximality and smoothness]}
Let  $\sigma$ be a  continuous section of the bundle $V$. Assume that  for all $x$ in $X$, $\sigma(x)$ does not belong to 	$W$.

Then denoting $\sigma=Z_\sigma+W_\sigma$, where $Z_\sigma$ is a section of $Z$ and $W_\sigma$ is a section of $W$, there exist a positive constants $A$ and $\lambda$ so that for all positive $t$
$$
\Vert \Phi_t(W_\sigma)\Vert\leq A\Vert\Phi_t( Z_\sigma\Vert ^{\lambda+1}\ .
$$	\end{lemma}
	\begin{proof}  Let us choose an auxiliary metric on $Z$ and $W$, since the flow is contracting on $Z$, we may reparametrise the flow so that for every $v$ in $Z$,
\begin{eqnarray}
	\Vert \Phi_t (v)\Vert=e^{-t}\Vert v\Vert\ . 	\label{eq:norm}
\end{eqnarray}
	Then the contraction property on $Z^*\otimes W$ tells that there is a positive  constants $B$  and $\lambda$ so that for all $w$ in  $Z^*\otimes W$ 
	$$
	\Vert \Phi_t(w)\Vert\leq B e^{-\lambda t} \Vert w\Vert\ .
	$$
	Thus for any $(u,v)$ in $Z\times W$, where $u$ is non zero, and all positive $t$,  
	$$
		\Vert \Phi_t(v)\Vert \leq   B e^{-\lambda t} \frac{\Vert v\Vert \cdotp \Vert \Phi_t(u)\Vert}{\Vert u\Vert} \leq A e^{-(\lambda+1)t} \Vert v\Vert \ 
	$$
Let now 
\begin{eqnarray*}
	C&\defeq& \inf \{\Vert Z_\sigma(x)\Vert\mid c\in X\}\ , \\
 	D&\defeq& \sup \{\Vert W_\sigma(x)\Vert\Vert\mid x\in X\}\ ,
\end{eqnarray*}
 	and recall that by hypothesis $C$ is positive.
 	Let now $z$ be a point in $x$, then for any positive $t$,
\begin{eqnarray*}
\Vert \Phi_t(W_\sigma(x))\Vert & \leq &B e^{-(\lambda+1)t} \Vert W_\sigma(x)\Vert\\
&\leq & B\left(\frac{\Vert \Phi_t(Z_\sigma(x))\Vert }{\Vert Z_\sigma(x)\Vert }\right)^{\lambda+1}\Vert W_\sigma(x)\Vert\\
&\leq &B DC^{-(\lambda+1)}
\Vert \Phi_t(Z_\sigma(x))\Vert ^{\lambda+1}\ .
		\end{eqnarray*}
Thus we obtain the lemma with $A=B DC^{-(\lambda+1)}$.			 \end{proof}

\subsubsection{Curves in bundles} The  limit maps $\Theta$ and $\overline\Theta$ then give rise to  two continuous,  flow invariant  maximal isotropic and  transverse  subbundles (also denoted $\Theta$ and $\overline\Theta$  of $E_\rho$.  We see these subbundles as sections, also denoted $\Theta$ and $\overline\Theta$ of  $\mathcal L_\rho$ the associated bundle over $X$ to the Grassmannian of totally isotropic planes $\bL$ in $E$

Let us choose an orientation on $\Sigma$ and thus a complex structure associated to the hyperbolic structure, as well as an orientation on $\partial_\infty\pi_1(\Sigma)$.
 
From hyperbolic geometry, we have a $\Gamma$ equivariant map  $h$ from  the unit bundle $\mathbf H^2$ of $\Sigma$ to $\partial_\infty\Gamma$, which associates to a unit vector $u$, the end point of the geodesic given by $Ju$, where $J$ is the complex structure on $\Sigma$ associated to the hyperbolic metric.

Thus we obtain a section $\sigma$ of $\mathcal L_\rho$ given by
$$
\sigma(u)\defeq \Theta(h(u))\ .
$$
For a representation close to be Fuchsian, since $\Theta(z)$ is transverse to $\overline\Theta(w)$ is $z\not=w$ by the Anosov property, we will consider $\sigma$ as  a section of the vector bundle 
$$
\mathcal T\defeq\T_{\Theta}\bL\subset \hom(\Theta,\overline\Theta)\ .
$$ that we freely  identify with $\ext^2(\Theta^*)$, using proposition \ref{pro:tgiso} by an identification that respects the lifts of the flow. Then we have 
\begin{proposition} 
	The decomposition $\mathcal T=Z\oplus W$ is a proximal vector bundle decomposition where 
	 \begin{eqnarray}
Z\defeq \ext^2\left(E_{p-1}^*\oplus  E_p^*\right)\ , & &
W\defeq\{\omega\in \ext^2(\Theta^*)\mid \left.\omega\right\vert_{E_p\oplus E_{p-1}}=0\}\ .
\end{eqnarray}

\end{proposition}
\begin{proof} By the Anosov property,   the flow contracts
 $$
 \ext^2(\Theta^*)=\bigoplus_{p\geq i>j}\ext^2(E_i^*\oplus E^*_j)
 \ ,
 $$ 
 and furthermore contracts less on $\ext^2(E^*_{p}\oplus E^*_{p-1})$ than on $\ext^2(E^*_i\oplus E^*_j)$ when $i<j$ and $j>p$.
	 \end{proof}
\begin{lemma}
For all $x$ in $X$, $\sigma(x)$ does no belong to $W_x$.
\end{lemma}  
\begin{proof} In the identification $\ext^2(\Theta)= \T\bL\subset\hom(\Theta,\overline\Theta)$, $W$ is a subset of 
$$
W_0\defeq \{f\mid \forall (u,v)\in E_{p-1}\times E_{p}\ . \ \  q(u,f(v))=q(v,f(u))=0\}
$$
But if $f\in W_0$, then 
$
f(E_{p})$ is included in $E^o_{p-1}\cap \overline\Theta
$.
Thus the graph of $f$ has an intersection of positive dimension with
$E_{p}\oplus (E^o_{p-1}\cap \overline\Theta)$.
It follows from the  third statement of proposition \ref{pro:trans} that $\sigma(x)$ does not belong $W$ for all $x$ in $X$.  \end{proof}
As a corollary of our proximal and smoothness lemma \ref{lem:prox} we now get, denoting $u$ the section of $Z$ of norm 1 which is so that  for all $t$, 
$$ u=\frac{\Phi_t(Z_\sigma(\varphi_{-t}x))}{\Vert\Phi_t(Z_\sigma(\varphi_{-t}x))\Vert}\ .$$ 
Observe that $u$ gives rise to an orientation of $Z$.
\begin{corollary}\label{cor:AWZ}
There exist positive constants $A$ and $\lambda$,  such that for all positive $t$, 
$$
\Vert \Phi_t(W_\sigma)\Vert\leq A\Vert\Phi_t( Z_\sigma\Vert ^{\lambda+1}\ .
$$
And in particular
$$
\lim_{t\to\infty}\left(\frac{\Phi_t(\sigma(\varphi_{-t}x))}{\Vert\Phi_t(\sigma(\varphi_{-t}x))\Vert}\right)=u \ .
$$
\end{corollary}

We now explain how this corollary implies the Smoothness Theorem \ref{theo:smooth}.
\begin{proof}[Proof of Theorem \ref{theo:smooth}]  
Let us also choose some auxiliary Riemannian metric and let $U$ be the unit vector field in
along $C\defeq\Theta(\partial_\infty\pi_1(\Sigma))$ so that $u_x$ is tangent at  $\Theta(x)$  to 
$$
Z_x\defeq \ext^2(E^*_{p-1}(x,y)\oplus E^*_{p-1}(x,y))\subset \T_{\Theta(x)}\bL\ ,
$$ 
and inducing the orientation compatible with that coming from $u$.

Observe that $u_x$ is continuous in $x$.  Then by corollary  \ref{cor:AWZ}
$$
\lim_{z\to x^+}\frac{\Theta(x)-\Theta(z)}{\Vert \Theta(x)-\Theta(z)\Vert}=u\ .
$$
Taking the opposite orientation on $\Sigma$ we obtain symmetrically
$$
\lim_{z\to x^-}\frac{\Theta(x)-\Theta(z)}{\Vert \Theta(x)-\Theta(z)\Vert}=-u\ .
$$
It follows that $C$ is a $C^1$ curve whose tangent space at $x$ is $Z_x$.
\end{proof}

\subsection{The last root flow and the Entropy Theorem}
We also have \cite[proposition 2.4]{BCLS} the following result
\begin{proposition}
	For $\rho$ with values in $\sop$ close to a Hitchin representation in $\mathsf{SO}(p,p-1)$, there exists a reparametrisation $\flo{\psi}$, called the {\em last root flow}, of $\flo{\varphi}$ so that the length of the closed orbit of $\psi$ associated to $\gamma$ is $\log\lambda_p(\rho(\gamma))+\log\lambda_{p-1}(\rho(\gamma))$.
\end{proposition}
Indeed, we observe the real line bundle $Z= \ext^2(E^*_{p-1}\oplus E^*_{p-1})$ is contracted by the flow and its {\em contraction spectrum} associates to the closed orbit $\gamma$ the number $\log\lambda_p(\rho(\gamma))+\log\lambda_{p-1}(\rho(\gamma))$.

The entropy theorem is now properly stated as
\begin{theorem}{\sc [Entropy Theorem]}\label{theo:entrop}
	For $\rho$ close enough to a Hitchin representation in $\mathsf{SO}(p,p-1)$, the entropy of the last root flow is equal to 1.
\end{theorem}
This theorem is also a consequence  to \cite{Pozzetti:2019uc}, also using a fundamental idea due to Potrie and Sambarino \cite{Potrie:2014uta}.

\begin{proof}
	We follow closely  Potrie and Sambarino \cite{Potrie:2014uta}, to obtain a proof of the Entropy Theorem \ref{theo:entrop}.
We observe that the if $\gamma_+,\gamma_-$ are respectively the attractive and repulsive fixed points of $\gamma$ on $\bgrf$, then $\Xi_q(\gamma_+,\gamma_-)$ is a fixed point of $\rho(\gamma)$ in $M_q$ and that its rate of contraction is given by $\lambda_p(\gamma)\cdotp\lambda_{p-1}(\gamma)$ on $\ext^2(E^*_{p}\oplus E^*_{p-1})$. 
The same discussion as in Potrie--Sambarino using SRB measures gives us the result in the neighbourhood of the Fuchsian representation by Corollary \ref{pro:trans} and Theorem \ref{theo:smooth} since the isotropic limit curve is $C^1$. Finally, as in \cite{Potrie:2014uta}, the analyticity of the entropy obtained in \cite{Bridgeman:2015ba} implies that the entropy is constant and equal to 1 on the neighbourhood of the Hitchin representations in $\mathsf{SO}(p,p-1)$. \end{proof}

It seems likely that all representations in the Hitchin component for $\mathsf{SO}(p,p)$ are Anosov, in which case the previous theorem applies to the whole Hitchin component.

\subsection{Entropy and the Affine Action Theorem}
\subsubsection{Affine group and quadratic forms}\label{para:inter} Let us consider a representation $\rho:\gamma\to\rho_\gamma$ of a surface group $\Gamma$ in the affine group of $E$ whose linear part $\rho^0$ is a Hitchin representation is in $\mathsf{SO}(p,p-1)$. We describe  the translation part by $\omega\in H^1_{\rho^0}(E)$, defined by the cocycle $\gamma\mapsto\omega_\gamma\defeq\rho(\gamma)(0)$.  Let $L$ be a one-dimensional vector space generated by a vector $f$.
Let $\braket{\mid}$ be the quadratic form on $E\oplus L$ , given by $\braket{u+xf \mid u+xf}=Q(u)-x^2$ of signature $(p,p)$.

The corresponding embedding of $\mathsf{SO}(p,p-1)$ to $\sop$ is so that have the $\mathsf{SO}(p,p-1)$ invariant decomposition
x$$
\T_{\operatorname{id}} \sop=\T_{\operatorname{id}} \mathsf{SO}(p,p-1)\oplus E\  .
$$
Accordingly,  we  consider  $\rep(\Gamma,\mathsf{SO}(p,p-1)$ as a subset of $\rep(\Gamma,\sop)$ and we identify  
$H^1_{\rho}(E)$ as a vector subspace of $\T_{\rho^0}\rep(\Gamma,\sop)$.

We know define this isomorphism more explicitly. We represent elements $\dot\rho$ in $\T_{\rho^0}\rep(\Gamma,\sop)$ by coholomogy class of cocycles 
$
\dot\rho:\gamma\mapsto\dot\rho_\gamma
$.
Let then $\mathcal H$ be the subset of $\T_{\rho^0}\rep(\Gamma,\sop)$ defined by 
\begin{eqnarray*}
	\mathcal H&\defeq&\{\dot\rho\mid \forall u,v\in E, \ \ \forall \gamma\in  \Gamma \ \ \braket{\dot\rho_\gamma(u)\mid v}=0\}\ ,\\
\end{eqnarray*}
\begin{proposition}{\sc [Interpretation]}\label{pro:iden} The map $\dot\rho\mapsto \omega$, where $\omega$ is defined by 
	$$
	\forall v\in E, \ \ Q(\omega_\gamma,v)=\braket{\dot\rho_\gamma(e_{f})\mid\rho(\gamma)(v)}
	$$
	is an isomorphism between $\mathcal H$ and $H^1_{\rho}(E)$.\end{proposition}

\subsubsection{Margulis invariant}
Let $\rho$ be a representation of $\Gamma$ in the affine  group, $\rho^0$ its linear part assumed to be a Hitchin representation in $\mathsf{SO}(p,p-1)$ and $\omega$ the affine deformation, that we see as a (closed) form in $\Omega^1(X,\cV_{0})$, where $\cV_{0}$ is the flat bundle on $X$  associated to $\rho^0$. By section \ref{sec:hitch}  we have the flow  invariant decomposition: 
$$
\cV_{0}=\sum_{i=1}^{2p-1}\cV_i\ .
$$
Let $\epsilon_p$ be the section of norm 1 of the spacelike line bundle $E_p$.  Let us choose a parametrisation of the geodesic flow, with generator $X$. Let $\mu$ be a measure invariant by the geodesic flow. 
We define as in \cite{Labourie:2001,Goldman:2008} the diffusion
$$
 M(\mu)\defeq\int_{\mathsf{U}S} Q(\epsilon_p, \omega(X))\  {\rm d} \mu \ .
 $$
One may notice that one could get rid of the choice of the parametrisation by working with invariant currents.

 Let $\flo{\rho}$ be a family of representations of $\Gamma$ in $\mathsf{SO}(p,p)$ associated to $\rho$, according to our Interpretation proposition \ref{pro:iden}, so that 
$$
\diffo \rho_t=\omega \ , \ \ \rho^0=\rho_0\ .
$$
For $t$ close to zero, $\rho_t$ is  close to a Hitchin representation in $\mathsf{SO}(p,p-1)$ (hence Borel Anosov in  $\sop$ by proposition \ref{anoano}) and thus also Borel Anosov \cite{Labourie:2006,Guichard:2012eg}. We can decompose the associated bundle as in Definition  \ref{def:ano} in $$
\cE_{\rho_t}=\bigoplus_{i=1}^p \cE^t_i\oplus \bigoplus_{i=1}^p \overline\cE^t_i \  .
$$ 
This decomposition is given by the limit curves. Since they depend  analytically on the representation \cite[Theorem 6.1]{Bridgeman:2015ba}, we may choose an identification of $\cE_{\rho_t}$ with $\cV_0\oplus L$, where $L$ is the trivial bundle such that furthermore
\begin{enumerate}
\item the quadratic form is constant,
\item the bundles $\cE^t_i$ and $\overline\cE^t_i$ are constant and thus denoted $\cE_i$ and $\overline\cE_i$
\item  Finally  $\cE_i=\cV_i$, $\overline \cE_i=\cV_{2p-i+1}$, for $i<p$, $\cE_p$ and $\overline \cE_p$ are the lightlike lines in $\cV_p\oplus L$.
\end{enumerate}
Let $\mu_\gamma$ be the  current represented to a closed orbit associated by a non trivial element $\gamma$ of $\Gamma$. The next lemma is a generalisation of  \cite[Lemma 2]{Goldman:2000a}.
\begin{lemma}\label{lem:marg}The variation of the eigenvalues are given as follows:
\begin{eqnarray*}
\diffo\lambda_p(\rho_t(\gamma))= \frac{1}{2}M(\mu_\gamma)\ & ,&  
\hbox{ for } i<p,\ \ \diffo \lambda_i(\rho_t(\gamma))= 0\ .
\end{eqnarray*}
\end{lemma}

\begin{proof}  We can obtain this lemma as a direct application of  \cite[Lemma 4.1.1]{Labourie:2018fj}, we reproduce the easy proof in this context. We choose basis $e_i$ and $\overline e_i$ of  $\cE_i$  and $\overline\cE_i$  respectively, so that  $
e_p=\frac{1}{2}(\epsilon_p+f)$, $\overline e_{p}=\frac{1}{2}(\epsilon_p-f)$. 
where $\epsilon_p$  is a basis of $E_p$ of norm $1$ and $f$ a basis of the trivial bundle $L$ of norm $-1$.
Writing $\dot a=\diffo a(t)$,  then $	M(\mu_\gamma)$ is equal to 
\begin{eqnarray*}
Q(\omega_\gamma,\epsilon_p)=Q(\dot\rho_\gamma(f) \epsilon_p)
	=\braket{e_p+\overline e_{p}\mid \dot\rho_\gamma(e_p- \overline e_{p})}=2\braket{\dot\rho_\gamma(e_p)\mid\overline e_p}=2\dot\lambda_p\ .
\end{eqnarray*}
Similarly, by proposition \ref{pro:trans}, $
0=\braket{\dot\rho_\gamma(\epsilon_k)\mid\epsilon_{2p-k}}$ for $k<p$, thus $\dot\lambda_k=0$. \end{proof}

\begin{corollary}\label{cor:marginv}
	For any measure $\mu$, if $g$ is the variation of the reparametrisation of the last root flow
\begin{eqnarray}
		2\int_{\mathsf UX}g\  {\rm d}\mu&=&M(\mu)\ .\label{eq:mivp}
\end{eqnarray}
\end{corollary}
\begin{proof} Let $\mu_\gamma$ is a current supported on a closed orbit. Then using the definition for the first equality, the fact that $\lambda_{p}=1$ for an $\mathsf{SO}(p,p-1)$ representation and lemma \ref{lem:marg}
$$
\int_{\mathsf UX}g\  {\rm d}\mu_\gamma=\frac{\dot\lambda_{p-1}}{\lambda_{p-1}}+\frac{\dot\lambda_{p}}{\lambda_{p}}=\dot\lambda_{p}=\frac{1}{2}M(\mu_\gamma)\ .	
$$	
Thus the equation \eqref{eq:mivp} holds  for all currents supported on closed orbits, hence for all linear combination of such by linearity, hence for all measures by density and continuity of the diffusion.
\end{proof}
\subsubsection{Abramov lemma} We will use the thermodynamic formalism and refer to \cite{Bridgeman:2015ba} for  a general discussion and references.  

Let $\flo{\psi^s}$ be  a family metric Anosov flow on a space $M$. Let $\ell^s_\gamma$ be the length of every closed orbit $\gamma$ for $\psi^s$. Let $f_s$ be a family of functions  on $M$ so that $$
\ell_\gamma^s=\int_0^{\ell^0_\gamma} f_s\circ\psi^0_u (x)\ {\rm d} u\ , 
$$
where $x$ is a point in $\gamma$ -- see \cite[Paragraph 3.1]{Bridgeman:2015ba} for details.   Then if $m$ is an invariant measure for $\flo{\psi^0}$, then 
\begin{eqnarray}
m^s\defeq \frac{1}{\int_M f_s m} f_s m\label{def:ms}\ ,
\end{eqnarray}
is an invariant measure for $\flo{\psi^s}$. The Abramov lemma  \cite[Lemma  2.4]{Sambarino:2014jv} is then that
\begin{eqnarray}
h(m^s)=	\frac{1}{\int_M f_s m} h(m)\ .
\end{eqnarray}
Then as a consequence.
\begin{lemma}\label{lem:abra}  Assume that $f^s(m)$ depends $C^1$ in $s$ and the derivative is bounded in $m$. Let $h_s$ be  the topological entropy of $\psi^s$, assuming that $h_s$ is constant and non zero  then 
\begin{eqnarray}
		0=\int_M 	\left.\frac{{\rm d} f_s}
	{{\rm d}s}
	\right\vert_{s=0}\  {\rm d}\mu_{\tiny{eq}}\ .
\end{eqnarray}
\end{lemma}
\begin{proof} Let us consider $\mu_{\tiny{eq}}^s$ as in formula \ref{def:ms}. By Abramov lemma  for the first equality and the definition of topological entropy for the second
$$
\frac{h_0}{\int_M f_s\  {\rm d}\mu_{\tiny{eq}}}=h(\mu_{\tiny{eq}}^s)\leq h_0\ .
$$
Thus for all $s$, $\int_M f_s  {\rm d}\mu_{\tiny{eq}} \geq 1$, while $\int_M f_0  {\rm d}\mu_{\tiny{eq}}=\int_M   {\rm d}\mu_{\tiny{eq}}=1$.
The result followss
\end{proof}

\subsubsection{Proof of the Affine action Theorem \ref{theo:aff}}
Let $\sigma$ be a Hitchin representation in  the affine group whose linear part $\rho$ is  Hitchin in $\mathsf{SO}(p,p-1)$. We interpret $\sigma$ as a a family of representation $(\rho_s)_{s\in[0,1]}$ in $\mathsf{SO}(p,p)$ with $\rho_0=\rho$ as in paragraph \ref{para:inter} and proposition \ref{pro:iden}.
 
Let $\psi^s$ be the last root flow of $\rho_s$ and $f_s$ be a family of reparametrisations of $\psi^0$ giving rise to $\psi_s$. Let $g=\left.\frac{{\rm d} f_s}
	{{\rm d}s}\right\vert_{s=0}$. Since the entropy of $\psi^s$ is constant by the Entropy Theorem \ref{theo:entrop}, we have by 
 lemma \ref{lem:abra} that $\int_g g\ {\rm d}\mu=0$ where $\mu$ is  the Bowen--Margulis measure of the last root flow of $\rho_0$. Thus by corollary \ref{cor:marginv},  $M(\mu)=0$.

Now by  \cite[Theorem 7.1 and Definition 4.4]{Ghosh:yPvRLZo8}  if there is a measure that annihilates the Margulis invariant, then the action on the affine group is not proper. This concludes the proof of Danciger and Zhang's Theorem \ref{theo:aff}  .

\bibliographystyle{amsplain}
\providecommand{\bysame}{\leavevmode\hbox to3em{\hrulefill}\thinspace}
\providecommand{\MR}{\relax\ifhmode\unskip\space\fi MR }
\providecommand{\MRhref}[2]{%
  \href{http://www.ams.org/mathscinet-getitem?mr=#1}{#2}
}
\providecommand{\href}[2]{#2}
\bibliographystyle{amsplain}

  \end{document}